\documentclass
[12pt]{article} \oddsidemargin 0in \evensidemargin 0in \textheight
9in \textwidth 6.5in \topmargin 0in \headheight 0in
\usepackage{amsmath,amsfonts,amsthm,amssymb}

\newtheorem{thm}{Theorem}
\newtheorem{theorem}{Theorem}[section]

\newtheorem{lemma}[theorem]{Lemma}

\newtheorem{defn}{Definition}

\newtheorem{exple}[theorem]{Example}

\newcommand{\N}{{\rm I} \hspace{-0.03in} {\rm N}}

\title{Contraction Principles in $M_s$-metric Spaces }

\author {N. Mlaiki$^{1}$, N. Souayah$^{2},$ K. Abodayeh$^{3},$ T. Abdeljawad$^{4}$
   \\ {\it Department of Mathematical Sciences, Prince Sultan University$^{1,}$$^{3,}$$^{4}$}
   \\  {\it Department of Natural Sciences, Community College, King Saud University$^{2}$}
    \\ {\it Riyadh, Saudi Arabia 11586 }
   \\ E-mail:    nmlaiki@psu.edu.sa$^{1}$\\
                 nsouayah@ksu.edu.sa$^{2}$\\
                 kamal@psu.edu.sa$^{3}$\\
                 tabdeljawad@psu.edu.sa$^{4}$ }

\date{}

\begin{document}
\maketitle{}
\begin{abstract}

In this paper, we give an interesting extension of the partial $S$-metric space which was introduced \cite{Nabil1} to the
$M_{s}$-metric space. Also, we prove the existence and uniqueness of a
fixed point for a self mapping on an $M_{s}$-metric space under different contraction principles.
\end{abstract}

\section{Introduction}

Many researchers over the years proved many interesting results on the  existence of a
fixed point for a self mapping on different types of metric spaces for
example, see (\cite{M2}, \cite{M3}, \cite{M4}, \cite{M5}, \cite{M6}, \cite{M7}, \cite{M8}, \cite{M9}, \cite{M10}, \cite{Postolache1}, \cite{Postolache2}, \cite{Postolache3}.)
The idea behind this paper was inspired by the work of Asadi in \cite{M1}.
He gave a more general extension of almost any metric space with two dimensions, and
that is not just by defining the self "distance" in a metric as in partial metric spaces \cite{Val2005,IATOP2010,IAFTPA2011, Marshad, TEH2012}, but
he assumed that is not necessary that the self "distance" is less than the value of the metric between two different elements.

In \cite{Nabil1}, an extension of $S$-metric
spaces to a partial $S$-metric spaces was introduced. Also, they showed that every
$S$-metric space is a partial $S$-metric space, but not every partial
$S$-metric space is an $S$-metric space. In our paper, we
introduce the concept of $M_{s}$- metric spaces which an extension of
the partial $S$-metric spaces in which we will prove some fixed point results.

First, we remind the reader of the definition of a partial $S$-metric space

\begin{defn}\cite{Nabil1}
Let $X$ be a nonempty set. A partial S-metric on $X$ is a
function $S_{p}:X^{3}\rightarrow [0,\infty)$ that satisfies the
following
conditions, for all $x,y,z,t \in X:$\\
\begin{itemize}
\item[(i)] $x=y$ if and only if $S_{p}(x,x,x)=S_{p}(y,y,y)=S_{p}(x,x,y)$,
\item[(ii)] $S_{p}(x,y,z)\le S_{p}(x,x,t)+S_{p}(y,y,t)+S_{p}(z,z,t)-S_{p}(t,t,t)$,
\item[(iii)] $S_{p}(x,x,x)\le S_{p}(x,y,z)$,
\item[(iv)] $S_{p}(x,x,y)=S_{p}(y,y,x).$
\end{itemize}
The pair $(X,S_{p})$ is called a partial S-metric space.
\end{defn}

Next, we give the definition of an $M_{s}$-metric space, but first we
introduce the following notation. \\

\noindent{\bf Notation.}
\begin{enumerate}
  \item $m_{{s}_{x,y,z}}:= \min \{ m_{s}(x,x,x),m_{s}(y,y,y),m_{s}(z,z,z) \}$
  \item $M_{{s}_{x,y,z}}:= \max \{ m_{s}(x,x,x),m_{s}(y,y,y),m_{s}(z,z,z) \}$
\end{enumerate}

\begin{defn} \label{def of Ms}
 An $M_s$-metric on a nonempty set X is
a function $m_{s} : X^{3} \rightarrow \mathbb{R}^{+}$ such that for all $x, y, z, t \in X$, the following conditions

are satisfied: \\
\begin{enumerate}
  \item $m_{s}(x,x,x)=m_{s}(y,y,y)=m_{s}(x,x,y)$ if and only if $x=y,$
  \item $m_{{s}_{x,y,z}} \le m_{s}(x,y,z),$
  \item $m_{s}(x,x,y)=m_{s}(y,y,x),$
  \item $(m_{s}(x,y,z)-m_{{s}_{x,y,z}})\le (m_{s}(x,x,t)-m_{{s}_{x,x,t}})+(m_{s}(y,y,t)-m_{{s}_{y,y,t}})+(m_{s}(z,z,t)-m_{{s}_{z,z,t}}).$
\end{enumerate}

The pair $(X,m_{s})$ is called an $M_{s}$-metric space.
 Notice that the condition $m_s(x,x,x)=m_s(y,y,y)=m_s(z,z,z)=m_s(x,y,z) \Leftrightarrow x=y=z$ implies (1) above.
\end{defn}
It is a straightforward to verify that every partial $S$-metric space is an $M_s$-metric space but the converse is not true. The following is an example of an $M_{s}$-metric which is not a partial $S$-metric space.

\begin{exple}
Let $X=\{1,2,3\}$ and define the $M_{s}$-metric space $m_{s}$ on $X$ by
$m_{s}(1,2,3)= 6 $,\\
$m_{s}(1,1,2)=m_{s}(2,2,1)=m_{s}(1,1,1)=8$,\\
$m_{s}(1,1,3)=m_{s}(3,3,1)=m_{s}(3,3,2)=m_{s}(2,2,3)=7$,\\
$m_{s}(2,2,2)=9, \text{   and   } m_{s}(3,3,3)=5$.\\
It is not difficult to see that $(X,m_{s})$ is an $M_{s}$-metric space, but since $m_{s}(1,1,1)\not\le m_{s}(1,2,3)$ we deduce that $m_{s}$ is not a partial $S$-metric space.

\end{exple}

\begin{defn}
Let $(X, m_{s})$ be an $M_s$-metric space. Then:
\begin{enumerate}
  \item A sequence $\{x_{n}\}$ in $X$ converges to a point $x$ if and
only if $$\lim_{n\rightarrow \infty}(m_{s}(x_{n},x_{n},x)-m_{{s}{x_{n},x_{n},x}})=0.$$
  \item A sequence $\{x_{n}\}$ in $X$ is said to be $M_{s}$-Cauchy sequence if and only if $$\lim_{n,m\rightarrow \infty}(m_{s}(x_{n},x_{n},x_{m})-m_{{s}{x_{n},x_{n},x_{m}}}), \text{   and    } \lim_{n\rightarrow \infty}(M_{{s}{x_{n},x_{n},x_{m}}}-m_{{s}{x_{n},x_{n},x_{m}}}) $$
exist and finite.
  \item An $M_s$-metric space is said to be complete if every $M_{s}$-Cauchy sequence $\{x_{n}\}$ converges to a point $x$ such
that $$\lim_{n\rightarrow
\infty}(m_{s}(x_{n},x_{n},x)-m_{{s}{x_{n},x_{n},x}})=0 \text{   and    } \lim_{n\rightarrow \infty}(M_{{s}{x_{n},x_{n},x}}-m_{{s}{x_{n},x_{n},x}})=0.$$
\end{enumerate}
\end{defn}
A ball in the $M_s-$metric $(X, m_s)$ space with center $x \in X$ and radius $\eta >0$ is defined by
$$ B_{s}[x,\eta]=\{ y\in X\mid m_{s}(x,x,y)-m_{{s}{x,x,y}}\le
\eta\}.$$
The topology of $(X, M_s)$ is generated by means of the basis $\beta = \{B_{s}[x,\eta]: \eta>0\}$.
\begin{lemma} \label{sequential}
Assume $x_n\rightarrow x$ and $y_n\rightarrow y$ as $n\rightarrow \infty$ in an $M_s-$metric space $(X, m_s)$. Then,
$$\lim_{n\rightarrow \infty} (m_s(x_n,x_n,y_n)-m_{sx_n,x_n,y_n})=m_s(x,x,y)-m_{sx,x,y}$$
\end{lemma}
\begin{proof}
The proof follows by the inequality $(4)$  in Definition \ref{def of Ms}. Indeed, we have

\begin{eqnarray}
\nonumber
  |(m_s(x_n,x_n,y_n)-m_{sx_n,x_n,y_n})-(m_s(x,x,y)-m_{sx,x,y})|  &\leq& 2[(m_s(x_n,x_n,x)-m_{sx_n,x_n,x})\\
   &+&  (m_s(y_n,y_n,y)-m_{sy_n,y_n,y})]
\end{eqnarray}

\end{proof}
\section{Fixed Point Theorems}

In this section, we consider some results about the existence and the uniqueness of  fixed point for  self mappings on an $M_{s}$-metric space,
under different contractions principles.

\begin{thm}\label{one}
Let $(X, m_{s})$ be a complete $M_s$-metric space and $T $ be a self mapping on $X$ satisfying the following
condition:
\begin{equation}\label{cont1}
  m_{s}(T x,T x, T y) \le k m_{s}(x, x,y),
\end{equation}
for all $x, y \in
X,$ where $k\in [0,1).$ Then T has a unique fixed point $u$. Moreover, $m_{s}(u,u,u)=0.$

\end{thm}
\begin{proof}
Since $k\in[0,1)$, we can choose a natural number $n_{0}$ such that
for a given $0<\epsilon <1,$ we have
$k^{n_{0}}<\dfrac{\epsilon}{8}.$ Let $T^{n_{0}}\equiv F$ and
$F^{i}x_{0}=x_{i}$ for all natural numbers $i$, where $x_{0} $ is arbitrary. Hence, for all $x,y \in X$, we have

$$
m_{s}(Fx,Fx,Fy)=m_{s}(T^{n_{0}}x,T^{n_{0}}x,T^{n_{0}}y)\le k^{n_{0}} m_{s}(x,x,y).
$$
For any $i$, we have
\begin{align*}
m_{s}(x_{i+1},x_{i+1},x_{i}) & = m_{s}(Fx_{i},Fx_{i},Fx_{i-1}) \\& \le
k^{n_{0}}m_{s}(x_{i},x_{i},x_{i-1}) \\& \le k^{n_{0} +i}m_{s}(x_{1},x_{1},x_{0})\rightarrow
0\ as \  i \rightarrow \infty.
\end{align*}
Similarly, by (\ref{cont1}) we have $m_s(x_i,x_i,x_i)\rightarrow 0$ as $i\rightarrow \infty$.
Thus, we choose $l$ such that
$$
m_{s}(x_{l+1},x_{l+1},x_{l})< \frac{\epsilon}{8}   \text{             and         }      m_{s}(x_{l},x_{l},x_{l})< \frac{\epsilon}{4}  .
$$
Now, let $\eta=\frac{\epsilon}{2}+m_{s}(x_{l},x_{l},x_{l})$. Define the set
$$ B_{s}[x_{l},\eta]=\{ y\in X\mid m_{s}(x_{l},x_{l},y)-m_{{s}{x_{l},x_{l},y}}\le
\eta\}.$$
Note that, $x_{l} \in B_{s}[x_{l},\eta]$. Therefore $B_{s}[x_{l},\eta]\neq \emptyset.$ Let $z\in
B_{s}[x_{l},\eta]$ be arbitrary. Hence,
\begin{align*}
m_{s}(Fz,Fz,Fx_{l})& \le k^{n_{0}} m_{s}(z,z,x_{l})\\ & \le
\frac{\epsilon}{8}[\frac{\epsilon}{2}+m_{{s}{z,z,x_{l}}}+m_{s}(x_{l},x_{l},x_{l})] \\ & <
\frac{\epsilon}{8}[1+2m_{s}(x_{l},x_{l},x_{l})].\end{align*}
Also, we know that $ m_{s}(Fx_{l},Fx_{l},x_{l})=m_{s}(x_{l+1},x_{l+1},x_{l})<
\dfrac{\epsilon}{8}.$
Therefore,
\begin{align*}
m_{s}(Fz,Fz,x_{l})-m_{{s}{Fz,Fz, x_{l}}}& \le
2[(m_{s}(Fz,Fz,Fx_{l})-m_{{s}{Fz,Fz,Fx_{l}}})]+(m_{s}(Fx_{l},Fx_{l},x_{l})-m_{{s}{Fx_{l},Fx_{l},x_{l}}})\\
& \le 2m_{s}(Fz,Fz,Fx_{l})+m_{s}(Fx_{l},Fx_{l},x_{l})]\\
&\le 2\frac{\epsilon}{8}(1+2m_{s}(x_{l},x_{l},x_{l}))+\frac{\epsilon}{8}\\
&= \frac{\epsilon}{4}+\frac{\epsilon}{8}+ \frac{\epsilon}{2} m_{s}(x_{l},x_{l},x_{l})\\
&< \frac{\epsilon}{2}+m_{s}(x_{l},x_{l},x_{l}).
\end{align*}

Thus, $Fz \in B_{b}[x_{l},\eta]$ which implies that $F$ maps
$B_{b}[x_{l},\eta]$ into itself. Thus, by repeating this process
we deduce that for all $n\ge 1$ we have $F^{n}x_{l}\in
B_{b}[x_{l},\eta]$ and that is $x_{m}\in B_{b}[x_{l},\eta]$
for all $m\ge l.$ Therefore, for all $m>n\ge l$ where $n=l+i$ for some $i$
\begin{align*}
m_{s}(x_{n},x_{n},x_{m})&=m_{s}(Fx_{n-1},Fx_{n-1},Fx_{m-1})\\
& \le k^{n_{0}}m_{s}(x_{n-1},x_{n-1},x_{m-1})\\
& \le  k^{2n_{0}}m_{s}(x_{n-2},x_{n-2},x_{m-2})\\
&\vdots \\
&\le k^{in_{0}}m_{s}(x_{l},x_{l},x_{m-i})\\
& \le m_{s}(x_{l},x_{l},x_{m-i})\\
&\le \frac{\epsilon}{2}+m_{{s}{x_{l},x_{l},x_{m-i}}}+m_{s}(x_{l},x_{l},x_{l})\\
& \le \frac{\epsilon}{2}+2m_{s}(x_{l},x_{l},x_{l}) .
\end{align*}
 Also, we have $m_{s}(x_{l},x_{l},x_{l})<\frac{\epsilon}{4},$ which implies that $m_{s}(x_{n},x_{n},x_{m})<\epsilon$ for all $m>n>l,$ and thus
 $m_{s}(x_{n},x_{n},x_{m})-m_{{s}{x_{n},x_{n},x_{m}}}<\epsilon$ for all
$m>n>l.$ By the contraction condition (\ref{cont1}) we see that the sequence $\{m_s(x_n,x_n,x_n)\}$ is decreasing and hence, for all
$m>n>l$,  we have
\begin{align*}
M_{{s}{x_{n},x_{n},x_{m}}}-m_{{s}{x_{n},x_{n},x_{m}}}&\le M_{{s}{x_{n},x_{n},x_{m}}} \\
& = m_{s}(x_{n},x_{n},x_{n}) \\
&\le k m_{s}(x_{n-1},x_{n-1},x_{n-1})\\
&\vdots \\
&\le k^{n}m_{s}(x_{0},x_{0},x_{0}) \rightarrow 0 \text{  as   } n\rightarrow \infty.
\end{align*}
Thus, we deduce that
$$
\lim_{n,m\rightarrow \infty}(m_{s}(x_{n},x_{n},x_{m})-m_{{s}{x_{n},x_{n},x_{m}}})=0, \text{   and    } \lim_{n\rightarrow \infty}(M_{{s}{x_{n},x_{n},x_{m}}}-m_{{s}{x_{n},x_{n},x_{m}}})=0 $$
Hence, the sequence $\{x_{n}\}$ is an $M_{s}$-Cauchy. Since
$X$ is complete, there exists $u\in X $ such that
$$
\lim_{n\rightarrow \infty} m_{s}(x_{n},x_{n},u)-m_{{s}{x_{n},x_{n},u}}=0,~~\lim_{n\rightarrow \infty} M_{sx_{n},x_{n},u}-m_{{s}{x_{n},x_{n},u}}=0
$$
The contraction condition (\ref{cont1}) implies that $m_s(x_n,x_n,x_n)\rightarrow 0$ as $n\rightarrow \infty$. Moreover, notice that
$$\lim_{n\rightarrow \infty} M_{sx_{n},x_{n},u}-m_{{s}{x_{n},x_{n},u}}=\lim_{n\rightarrow \infty}|m_s(x_n,x_n,x_n)-m_s(u,u,u)|=0,$$ and hence  $m_s(u,u,u)=0$.
Since $x_n\rightarrow u$, $m_s(u,u,u)=0$ and $m_s(x_n,x_n,x_n)\rightarrow 0$ as $n\rightarrow \infty$  then $\lim_{n\rightarrow\infty}m_s(x_n,x_n,u)= \lim_{n\rightarrow\infty} m_{sx_n,x_n,u}=0$.
Since $m_s(Tx_n,Tx_n,Tu)\leq k m_s(x_n,x_n,u)\rightarrow 0 $ as   $n\rightarrow \infty$, then $Tx_n\rightarrow Tu$.

Now, we show that $Tu=u.$ By Lemma \ref{sequential} and that $Tx_n\rightarrow Tu$ and $x_{n+1}=Tx_n\rightarrow u$, we have

\begin{align*}
\lim_{n\rightarrow \infty} m_{s}(x_{n},x_{n},u)-m_{{s}{x_{n},x_{n},u}}&=0\\
&= \lim_{n\rightarrow \infty} m_{s}(x_{n+1},x_{n+1},u)-m_{{s}{x_{n+1},x_{n+1},u}}\\
&=\lim_{n\rightarrow \infty} m_{s}(Tx_{n},Tx_{n},u)-m_{{s}{Tx_{n},Tx_{n},u}} \\
&= m_{s}(u,u,u)-m_{{s}{Tu,Tu,u}}\\
&= m_{s}(Tu,Tu,u)-m_{{s}{Tu,Tu,u}}.
\end{align*}
Hence, $m_{s}(Tu,Tu,u)= m_{{s}{Tu,Tu,u}}=m_s(u,u,u)$ , but also by the contraction condition (\ref{cont1}) we see that  $m_{{s}{Tu,Tu,u}}=m_s(Tu,Tu,Tu)$ . Therefore, (2) in  Definition \ref{def of Ms}  implies that $Tu=u.$ \\

To prove the uniqueness of the fixed point $u,$
assume that $T$ has two fixed points $u,v \in X;$ that is, $Tu=u$
and $Tv=v.$ Thus,
$$m_{s}(u,u,v)=m_{s}(Tu,Tu,Tv)\le km_{s}(u,u,v)<m_{s}(u,u,v),$$

$$m_{s}(u,u,u)=m_{s}(Tu,Tu,Tu)\le km_{s}(u,u,u)<m_{s}(u,u,u),$$

and
$$m_{s}(v,v,v)=m_{s}(Tv,Tv,Tv)\le km_{s}(v,v,v)<m_{s}(v,v,v),$$

which implies that $m_{s}(u,u,v)=0=m_s(u,u,u)=m_s(v,v,v),$ and hence $u=v$ as desired.
Finally,assume that $u$ is a fixed point of $T.$ Then applying the contraction condition (\ref{cont1}) with $k \in [0,1)$, implies that
\begin{align*}
m_{s}(u,u,u) &=m_{s}(Tu,Tu,Tu)\\
&\le k m_{s}(u,u,u)\\
&\vdots \\
&\le k^n m_{s}(u,u,u)
\end{align*}
Taking the limit as $n$ tends to infinity, implies that is $m_{s}(u,u,u)=0.$
\end{proof}

In the following result, we prove the existence and uniqueness of a fixed point for a self mapping in $M_{s}$-metric space, but under a more general contraction.
\begin{thm}\label{two}
Let $(X, m_{s})$ be a complete $M_{s}$-metric space and $T $ be a self mapping on $X$ satisfying the following
condition:
\begin{equation}\label{cont2}
  m_{s}(T x,T x, T y) \le \lambda[ m_{s}(x,x, Tx)+ m_{s}(y,y,Ty)],
\end{equation}
for all $ x, y \in X,$ where $\lambda \in [0,\frac{1}{2}). $ Then T has a unique fixed point $u,$ where $m_{s}(u,u,u)=0.$
\end{thm}

\begin{proof}
Let $x_{0}\in X$ be arbitrary. Consider the sequence $\{x_{n}\}$
defined by $x_{n}=T^{n}x_{0}$ and $m_{s_{n}}=m_{s}(x_{n},x_{n},x_{n+1}).$
Note that if there exists a natural number $n$ such that
$m_{s_{n}}=0,$ then $x_{n}$ is a fixed point of $T$ and we are done.
So, we may assume that $m_{s_{n}}>0,$ for $n\ge 0.$ By
(\ref{cont2}), we obtain for any $n\ge 0,$
\begin{align*}
m_{s_{n}}& =m_{s}(x_{n},x_{n},x_{n+1})=m_{s}(Tx_{n-1},Tx_{n-1},Tx_{n})\\
&\le \lambda[ m_{s}(x_{n-1},x_{n-1}, Tx_{n-1})+ m_{s}(x_{n},x_{n}, Tx_{n})] \\
& = \lambda[ m_{s}(x_{n-1},x_{n-1}, x_{n})+ m_{s}(x_{n},x_{n}, x_{n+1})]\\
& = \lambda [m_{s_{n-1}}+m_{s_{n}}].
\end{align*}
Hence, $m_{s_{n}}\le \lambda m_{s_{n-1}}+ \lambda m_{s_{n}},$ which implies $m_{s_{n}}\le \mu
m_{s_{n-1}},$ where $\mu=\frac{\lambda}{1 -\lambda}<1$ as $\lambda
\in [0,\frac{1}{2}).$ By repeating this process we get
$$m_{s_{n}}\le \mu^{n} m_{s_{0}}.$$ Thus, $\lim_{n\rightarrow
\infty}m_{s_{n}}=0.$
By (\ref{cont2}), for all natural numbers
$n,m$ we have
\begin{align*}
m_{s}(x_{n},x_{n},x_{m})& =m_{s}(T^{n}x_{0},T^{n}x_{0},T^{m}x_{0})=m_{s}(Tx_{n-1},Tx_{n-1}, Tx_{m-1})\\
&\le \lambda[ m_{s}(x_{n-1},x_{n-1}, Tx_{n-1})+ m_{s}(x_{m-1},x_{m-1}, Tx_{m-1})] \\
& = \lambda[ m_{s}(x_{n-1},x_{n-1}, x_{n})+ m_{s}(x_{m-1},x_{m-1}, x_{m})]\\
& = \lambda [m_{s_{n-1}}+m_{s_{m-1}}].
\end{align*}
Since  $\lim_{n\rightarrow
\infty}m_{s_{n}}=0,$ for every $\epsilon>0$, we can find a natural
number $n_{0}$ such that $m_{s_{n}}<\frac{\epsilon}{2}$ and
$m_{s_{m}}<\frac{\epsilon}{2}$ for all $m,n >n_{0}.$ Therefore, it
follows that
$$
m_{s}(x_{n},x_{n},x_{m})\le \lambda [m_{s_{n-1}}+m_{s_{m-1}}]< \lambda[\frac{\epsilon}{2}+\frac{\epsilon}{2}]<
\frac{\epsilon}{2}+\frac{\epsilon}{2}=\epsilon\ \   \text{for all }
n,m>n_{0}.
$$
This implies that
$$
m_{s}(x_{n},x_{n},x_{m})-m_{{s}{x_{n},x_{n},x_{m}}}<\epsilon \ \   \text{for all }
n,m>n_{0}.
$$
Now, for all natural numbers $n,m$ we have
\begin{align*}
M_{{s}{x_{n},x_{n},x_{m}}}& =m_{s}(Tx_{n-1},Tx_{n-1}, Tx_{n-1})\\
&\le \lambda[ m_{s}(x_{n-1},x_{n-1}, Tx_{n-1})+ m_{s}(x_{n-1}, x_{n-1}, Tx_{n-1})] \\
& = \lambda[ m_{s}(x_{n-1}, x_{n-1}, x_{n})+ m_{s}(x_{n-1},x_{n-1},  x_{n})]\\
& = \lambda [m_{s_{n-1}}+m_{s_{n-1}}]\\
& = 2\lambda m_{s_{n-1}}.
\end{align*}
As  $\lim_{n\rightarrow \infty}m_{s_{n-1}}=0,$ for every $\epsilon>0$ we can find a natural
number $n_{0}$ such that $m_{s_{n}}<\frac{\epsilon}{2}$ and
for all $m,n >n_{0}.$ Therefore, it
follows that $$ M_{{s}{x_{n},x_{n},x_{m}}}\le \lambda
[m_{s_{n-1}}+m_{s_{n-1}}]<
\lambda[\frac{\epsilon}{2}+\frac{\epsilon}{2}]<
\frac{\epsilon}{2}+\frac{\epsilon}{2}=\epsilon\ \   \text{for all }
n,m>n_{0},$$ which implies that
$$M_{{s}{x_{n},x_{n},x_{m}}}-m_{{s}{x_{n},x_{n},x_{m}}}<\epsilon \ \   \text{for all }
n,m>n_{0}.$$
Thus, $\{x_{n}\}$ is an $M_{s}$-Cauchy sequence in $X.$
Since $X$ is complete, there exists $u \in X $ such that
$$
\lim_{n\rightarrow\infty} m_{s}(x_{n},x_{n},u)-m_{{s}{x_{n},x_{n},u}}=0.
$$
Now, we show that $u$ is a fixed point of $T$ in $X.$ For any natural number $n$ we have,
\begin{align*}
\lim_{n\rightarrow \infty} m_{s}(x_{n},x_{n},u)-m_{{s}{x_{n},x_{n},u}}&=0\\
&= \lim_{n\rightarrow \infty} m_{s}(x_{n+1},x_{n+1},u)-m_{{s}{x_{n+1},x_{n+1},u}}\\
&=\lim_{n\rightarrow \infty} m_{s}(Tx_{n},Tx_{n},u)-m_{{s}{Tx_{n},Tx_{n},u}} \\
&= m_{s}(Tu,Tu,u)-m_{{s}{Tu,Tu,u}}.
\end {align*}

This implies that $m_{s}(Tu,Tu,u)-m_{{s}{u,u,Tu}}=0,$ and that is $m_{s}(Tu,Tu,u)= m_{{s}{u,u,Tu}}.$
Now, assume that $$m_{s}(Tu,Tu,u)=m_{s}(Tu,Tu,Tu)\le 2\lambda m_{s}(u,u,Tu)= 2\lambda m_{s}(Tu,Tu,u)<  m_{s}(u,u,Tu).$$
Thus, $$m_{s}(Tu,Tu,u)=m_{s}(u,u,u)\le m_{s}(Tu,Tu,Tu)\le 2\lambda m_{s}(u,u,Tu)<m_{s}(u,u,Tu)$$
Therefore, $Tu=u$ and thus $u$ is a fixed point of $T.$ \\

Next, we show that if $u$ is a fixed point, then $m_{s}(u,u,u)=0.$ Assume that $u$
is a fixed point of $T,$ then using the contraction (\ref{cont2}), we have
\begin{align*} m_{s}(u,u,u)
&=m_{s}(Tu,Tu,Tu)\\&\le \lambda [m_{s}(u,u,Tu)+m_{s}(u,u,Tu)]\\
&=2\lambda m_{s}(u,u,Tu)\\
&=2\lambda m_{s}(u,u,u)\\
&<m_{s}(u,u,u)\ \   \text{since }
\lambda \in [0,\frac{1}{2}); \end{align*} that is, $m_{s}(u,u,u)=0.$

Finally, To prove uniqueness, assume that $T$ has two fixed points say $u,v \in
X.$ Hence,
$$
m_{s}(u,u,v)=m_{s}(Tu,Tu,Tv)\le \lambda [m_{s}(u,u,Tu)+m_{s}(v,v,Tv)]=\lambda [m_{s}(u,u,u)+m_{s}(v,v,v)]=0,
$$
which implies that $m_{s}(u,u,v)=0=m_{s}(u,u,v)=m_{s}(u,u,v),$ and hence $u=v$ as required. \\

\end{proof}

In closing, the authors would like to bring to the reader's attention that in this interesting
$M_{s}$-metric space it is possible to add some control functions in both contractions of Theorems \ref{one}, and \ref{two}.

\begin{thm}\label{three}
Let $(X, m_{s})$ be a complete $M_s$-metric space and $T $ be a self mapping on $X$ satisfying the following
condition: for all $x,y,z\in X$
\begin{equation}\label{1}
m_s (Tx,Ty,Tz)\leq m_s (x,y,z)-\phi (m_s (x,y,z)),
\end{equation}
where $\phi :[0,\infty)\rightarrow[0,\infty)$ is a continuous and non-decreasing function and $\phi^{-1}(0)=0$ and $\phi (t)>0$ for all $t>0$. Then $T$ has a unique fixed point in $X$.
\end{thm}
\begin{proof}
Let $x_0 \in X$. Define the sequence $\{x_{n}\}$ in $X$ such that $x_n =T^{n-1}x_0 =Tx_{n-1}$, for all $n\in \N$. Note that if there exists an $n\in\N$ such that $x_{n+1}=x_n$, then $x_n$ is a fixed point for $T$. Without lost of generality, assume that $x_{n+1}\neq x_n$, for all $n\in\N$. Now

\begin{eqnarray}\label{eq2}
  m_s (x_n ,x_{n+1}, x_{n+1}) &=& m_s (Tx_{n-1},Tx_n ,Tx_n )\nonumber \\
   &\leq & m_s (x_{n-1} ,x_{n}, x_{n}) -\phi (m_s (x_{n-1} ,x_{n}, x_{n})) \nonumber \\
   &\leq & m_s (x_{n-1} ,x_{n}, x_{n}).
\end{eqnarray}

Similarly, we can prove that $m_s (x_{n-1} ,x_{n}, x_{n}) \leq m_s (x_{n-2} ,x_{n-1}, x_{n-1})  .$ Hence, $m_s (x_n ,x_{n+1}, x_{n+1})$ is a monotone decreasing sequence. Hence there exists $r\geq 0$ such that
$$
\lim_{n\rightarrow\infty} m_s (x_n ,x_{n+1}, x_{n+1}) =r.
$$
Now, by taking the limit as $n\rightarrow\infty$ in the inequality (\ref{eq2}), we get $r\leq r-\phi (r)$ which leads to a contradiction unless $r=0$.
Therefore,
$$
\lim_{n\rightarrow\infty} m_s (x_n ,x_{n+1}, x_{n+1}) =0. \\
$$
Suppose that  $\{x_{n}\}$ is not an $M_s$-Cauchy sequence. Then there exists an $\epsilon>0$ such that we can find subsequences $x_{m_k}$ and $x_{n_k}$ of $\{x_{n}\}$ such that
\begin{equation}\label{eq3}
 m_s (x_{n_k} ,x_{m_k} ,x_{m_k} )-m_{sx_{n_k} ,x_{m_k} ,x_{m_k}} \geq \epsilon.
\end{equation}
 Choose $n_k$ to be the smallest integer with $n_k >m_k$ and satisfies the inequality (\ref{eq3}).
 Hence, $m_s (x_{n_k} ,x_{m_{k-1}} ,x_{m_{k-1}} )-m_{sx_{n_k} ,x_{m_{k-1}} ,x_{m_{k-1}}} < \epsilon .$ \\
Now,
\begin{eqnarray*}
  \epsilon &\leq & m_s (x_{m_k} ,x_{n_k} ,x_{n_k} )-m_{sx_{m_k} ,x_{n_k} ,x_{n_k}} \\
    &\leq & m_s (x_{m_k} ,x_{n_{k-1}} ,x_{n_{k-1}} )+2m_s (x_{n_{k-1}} ,x_{n_{k-1}} ,x_{n_{k-1}} )-m_{sx_{m_k} ,x_{n_{k-1}} ,x_{n_{k-1}}} \\
    &\leq & \epsilon+2m_s (x_{n_{k-1}} ,x_{n_{k-1}} ,x_{n_{k-1}} ) \\
   &< & \epsilon,
\end{eqnarray*}
as $n\rightarrow\infty$. Hence, we have a contradiction. \\
Without lost of generality, assume that $m_{sx_n ,x_n ,x_m}= m_s (x_n ,x_n ,x_n)$. Then we have
\begin{eqnarray*}
  0\leq m_{sx_n ,x_n ,x_m}-m_{sx_n ,x_n ,x_m} &\leq & M_{sx_n ,x_n ,x_m} \\
   &=& m_s (x_n ,x_n ,x_n) \\
   &=& m_s (Tx_{n-1},Tx_{n-1},Tx_{n-1}) \\
   &\leq & m_s (x_{n-1},x_{n-1},x_{n-1})-\phi (m_s (x_{n-1},x_{n-1},x_{n-1}))  \\
   &\leq & m_s (x_{n-1},x_{n-1},x_{n-1}) \\
   &\vdots&  \\
   &\leq& m_s (x_{0},x_{0},x_{0})
\end{eqnarray*}
Hence, $\displaystyle\lim_{n\rightarrow\infty}m_{sx_n ,x_n ,x_m}-m_{sx_n ,x_n ,x_m}$ exists and finite. Therefore, $\{x_{n}\}$ is an $M_s$-Cauchy sequence.
Since $X$ is complete, the sequence $\{x_{n}\}$ converges to an element $x\in X$; that is,
\begin{eqnarray*}
  0 &=&\lim_{n\rightarrow\infty} m_s (x_n ,x_n ,x)-m_{sx_n ,x_n ,x}  \\
   &=& \lim_{n\rightarrow\infty} m_s (x_{n+1} ,x_{n+1} ,x)-m_{sx_{n+1} ,x_{n+1} ,x} \\
   &=& \lim_{n\rightarrow\infty} m_s (Tx_n,Tx_n,x)-m_{sTx_n,Tx_n,x} \\
   &=& m_s (Tx,Tx,x)-m_{sTx,Tx,x} .
\end{eqnarray*}
Similarly to the proof of Theorem \ref{two}, it is not difficult to show that this implies that,  $Tx=x$ and so $x$ is a fixed point. \\
Finally, we show that $T$ has a unique fixed point. Assume that there are two fixed points $u,v\in X$ of $T$. If we have $m_s (u,u,v)>0$, then Condition (\ref{1}) implies that
$$
m_s (u,u,v)=m_s (Tu,Tu,Tv)\leq m_s (u,u,v)-\phi (m_s (u,u,v))<m_s (u,u,v),
$$
and that is a contradiction. Therefore, $m_s (u,u,v)=0$ and similarly $m_s (u,u,u)=M_s (v,v,v)=0$ and thus $u=v$ as desired.
\end{proof}

In closing, is it possible to define the same space but without the symmetry condition, (i.e. $m_{s}(x,x,y)\neq m_{s}(y,y,x)?$) If possible, what kind of results can be obtained in such space?


\begin{thebibliography}{999}


\bibitem{M1} M. Asadi et al., \textit{New extension of $p$-metric spaces with some fixed point results on $M$-metric spaces,}
Journal of Inequalities and Applications, (2014), {2014:18}.
\bibitem{M2} S. Matthews,  \textit{Partial metric topology,}
 Ann. NY. Acad. Sci., (1994) {\bf{728}}, 183-197.
\bibitem{M3} S. Shukla: \textit{Partial b-metric spaces and fixed point theorems,}
 Mediterranean Journal of Mathematics, {\bf{11}}, 703-711,
{2014}, {2014:18}
\bibitem{Nabil1} N. Mlaiki, \textit{A contraction principle in partial S-metric space,} Universal journal of mathematics and mathematical.  {\bf5}(2).
               (2014) 109-119
\bibitem{M4} N. Mlaiki,  \textit{$\alpha$-$\psi$-contractive mapping on S-metric space}, Mathematical Sciences Letters, \textbf{4} (2015), 9--12.

\bibitem{M5} N. Mlaiki, \textit{Common fixed points in complex S-metric space}, Advances in Fixed Point Theory, \textbf{4} (2014), 509--524.

\bibitem{M6} A. Mukheimer, \textit{$\alpha$-$\psi$-$\phi$-contractive mappings in ordered partial $b$-metric spaces}, Journal of  Nonlinear Sciences and Applications, {\bf 7} (2014), 168-–179.

\bibitem{M7} T. Abdeljawad, E. Karapinar and K. Ta\c{s}, \textit{A
generalized contraction principle with control functions on partial
metric spaces},Computer and Mathematics with Applications, 63 (3) (2012), 716-719.

\bibitem{M8}T. Abdeljawad, \textit{Fixed points for generalized weakly
contractive mappings in partial metric spaces}, Math. Comput.
Modelling  54  (11-12) (2011), 2923--2927.

\bibitem{M9}T. Abdeljawad, \textit{ Meir-Keeler alpha-contractive fixed and common fixed point theorems}, Fixed point theory and applications, article number 19   DOI:10.1186/1687-1812-2013-19.

\bibitem{TEH2012}T. Abdeljawad,   E. Karapinar, H. Aydi, \textit{A new Meir-Keeler type coupled fixed point on ordered partial metric spaces,} Mathematical Problems in Engineering, Vol. 2012, Article ID 327273, 20 pages, 2012. doi:10.1155/2012/327273.

\bibitem{M10}W. Shatanawi, and P. Ariana, \textit{Some coupled fixed point theorems in quasi-partial metric spaces}, Fixed point theory and applications   Article Number: 153   DOI: 10.1186/1687-1812-2013-153   Published: 2013.

\bibitem{Val2005} O. Valero, \textit{On Banach fixed point theorems for partial metric spaces,}
Applied General Topology, vol. 6, no. 2, pp. 229-240, 2005.


\bibitem{IATOP2010}I. Altun, F. Sola, and H. Simsek, \textit{Generalized contractions on partial metric spaces,}
Topology and Its Applications, vol. 157, no. 18, pp. 2778-2785, 2010.

\bibitem{IAFTPA2011} I. Altun and A. Erduran, \textit{Fixed point theorems for monotone mappings on partial metric spaces,} Fixed Point Theory and Applications, vol. 2011, Article ID 508730, 10 pages, 2011. doi:10.1155/2011/508730

\bibitem{Marshad}A. Shoaib, M. Arshad, and J. Ahmad, \textit{Fixed point results of locally contractive mappings in ordered quasi-partial metric spaces,} Scientific World Journal. 2013; 2013: 194897.



\bibitem{Postolache1} M.A. Miandaragh, M. Postolache, and S. Rezapour, \textit{Some approximate fixed point results for generalized $\alpha$-contractive mappings,} Sci. Bull. (Politeh.) Univ. Buchar., Ser. A, Appl. Math. Phys., Vol. 75 (2), pp. 3-10, 2013.

\bibitem{Postolache2} W. Shatanawi and M. Postolache, \textit{Some Fixed-Point Results for a G-Weak Contraction in G-Metric Spaces,} Abstract and Applied Analysis, Vol. 2012, Article ID 815870, 19 pages DOI: 10.1155/2012/815870, 2012.

\bibitem{Postolache3} W. Shatanawi, M. Postolache and Z. Mustafa, \textit{Tripled and coincidence fixed point theorems for contractive mappings satisfying $\Phi$-maps in partially ordered metric spaces,} An. S¸t. Univ. Ovidius Constant. a, Vol. 22(3), pp. 179 - 203, 2014.

\end{thebibliography}
\end{document}